\documentclass[12pt]{amsart}

\usepackage{amssymb,amsmath,amsfonts,euscript,amscd,mathrsfs,amsthm,graphicx}
\usepackage[all,cmtip,line]{xy}

\DeclareMathOperator{\F}{\mathbb{F}}

\DeclareMathOperator{\Z}{\mathbb{Z}}

\DeclareMathOperator{\Ext}{Ext}

\DeclareMathOperator{\Id}{Id}

\DeclareMathOperator{\Ker}{Ker}

\DeclareMathOperator{\Res}{Res}

\DeclareMathOperator{\Syl}{Syl}

\DeclareMathOperator{\Tr}{Tr}

\theoremstyle{definition}
\newtheorem{thm}{Theorem}[section]

\newtheorem{prop}[thm]{Proposition}
\newtheorem{lemma}[thm]{Lemma}

\newtheorem*{remark}{Remark}

\newcommand{\lr}[1]{\langle #1 \rangle}

\begin{document}

\title{Computations in Hochschild Cohomology of Group Algebras}
\author{Adam Allan}
\address{Dept. of Mathematics, St. Louis University, St. Louis, MO 63103}
\email{aallan@slu.edu}


\date{19 October 2011}

\begin{abstract}The Hochschild cohomology ring of a group algebra is an object that has received recent attention, but is difficult to compute, in even the simplest of cases. In this paper, we use the product formula due to Witherspoon and Siegel to extend some of their computations. In particular, we compute the Hochschild cohomology algebra of group algebras $kG$ where $|G| \leq 15$ and we provide an alternative computation of the ring $HH^*(k(E \ltimes P))$ considered by Kessar and Linckelmann.\end{abstract}

\maketitle

\section{Background}

We suppose throughout this paper that $G$ is a finite group and $k$ a field of characteristic $p$. The Hochschild cohomology ring $HH^*(kG)$ is defined as $\Ext_{kG \otimes kG^{\text{op}}}^*(kG,kG)$ where $kG$ has the obvious bimodule structure. Using the Eckmann-Shapiro Lemma, one sees that $HH^*(kG)$ and $H^*(G,{_{\psi}kG})$ are isomorphic as algebras, where $G$ acts on $kG$ via conjugation. Furthermore, $H^*(G,{_{\psi}kG})$ has a well-known additive decomposition. More precisely, let $\{ g_i \}$ be representatives of the conjugacy classes of $G$, $H_i = C_G(g_i)$, and $W_i$ the subspace of $kG$ spanned by all conjugates of $g_i$, so that $W_i \simeq k_{H_i}\uparrow^G$. Since ${_{\psi}kG} = \bigoplus W_i$, the Eckmann-Shapiro Lemma yields a linear isomorphism

$$H^*(G,{_{\psi}kG}) = \bigoplus \Ext_{kG}^*(k,W_i) \simeq \bigoplus \Ext_{kH_i}^*(k,k) = \bigoplus H^*(H_i)$$

These results were known for some time, but it was only in \cite{Holm} and \cite{Cibils} that the case of $HH^*(kG)$ for $G$ an abelian group was fully completed. Using bimodule resolutions of $kG$ as a $G$-module, one may compute $HH^*(kG)$ for several additional classes of groups, including the quaternions. However, an alternative method of computation is provided in \cite{Witherspoon}. There, it is determined how the additive decomposition behaves with respect to the cup product of $H^*(G,{_{\psi}kG})$. In fact, they derive a product formula

\begin{equation}\label{ProductFormula}
\gamma_i(\alpha) \smile \gamma_j(\beta) = \sum_{x \in D} \gamma_k(\Tr_W^{H_k}(\Res^{{^y H_i}}_W y^*\alpha \smile \Res^{{^{yx}H_j}}_W (yx)^*\beta))
\end{equation}

Here, $\gamma_i$ is the composite $H^*(H_i) \rightarrow H^*(H_i,{_{\psi}kG}) \rightarrow H^*(G,{_{\psi}kG})$, $D \subseteq G$ satisfies $G = \coprod_{x \in D} H_i x H_j$, for each $x \in D$ we have $g_k$ conjugate with $g_i{^x g_j}$ with $g_k = {^y(g_i{^xg_j})}$, and $W = {^yH_i} \cap {^{yx}H_j}$. They used this product formula, in particular to find finite presentations of the algebras $HH^*(kG)$ for several small examples and for dihedral 2-groups. The aim of this paper is to extend these computations in several directions.

The paper will proceed as follows. In section 2 we will introduce a modified version of $HH^*(kG)$ that is more convenient for computations but leaves unaltered the fundamental aspects of $HH^*(kG)$ and the product formula. Basically, we will modify the degree zero elements to ignore the redundant information they contain. In section 3 we provide the promised alternative proof to Kessar and Linckelmann's result from \cite{LinckelmannFrobenius} using the product formula. In section 4 we proceed to compute finite presentations for $HH^*(kG)$ where $|G| \leq 15$. Four of these cases were handled by the work performed in \cite{Witherspoon} and so this section is a continuation of the work started there.

\section{Preliminary Results}

There is often times much extraneous information given by $HH^0(kG)$. This occurs when $H_i$ is a $p'$-group, since in this case $H^*(H_i) = k$ concentrated in degree zero. It is therefore useful to consider the Tate analogue $\widehat{HH}^*(kG)$ of Hochschild cohomology. For an algebra $A$, as per \cite{LinckelmannFrobenius} we define

$$\widehat{HH}^*(A) = \bigoplus \overline{\text{Hom}}_{A \otimes A^{\text{op}}}(\Omega^n_{A \otimes A^{\text{op}}}(A),A)$$

where $\overline{\text{Hom}}_{A \otimes A^{\text{op}}}$ denotes the homomorphism space in the stable module category ${_{A \otimes A^{\text{op}}}\overline{\text{mod}}}$, and $\Omega$ denotes the Heller operator. Recall that if $E_{A-A}(A)$ denotes the ring of all $A \otimes A^{\text{op}}$-linear endomorphisms of $A$, then $Z(A) \simeq E_{A-A}(A)$ under the map $z \mapsto (a \mapsto az))$. There is an ideal $Z^{\text{pr}}(A)$ of $Z(A)$ defined as the set of all endomorphisms of $A$ that factor through an $A \otimes A^{\text{op}}$-projective module. Then the stable center $\bar{Z}(A)$ is defined as $Z(A)/Z^{\text{pr}}(A)$. It is easy to see that $\widehat{HH}^n(A) = HH^n(A)$ for $n> 0$ and that $\widehat{HH}^0(A) = \bar{Z}(A)$. Moreover, $\widehat{HH}^*(A)$ satisfies the analogue of Tate duality in the sense that $\widehat{HH}^{-n}(A) \simeq HH^{n-1}(A)^*$ for $n > 0$, where $^*$ denotes the dual space. The next proposition justifies the assertion that $\widehat{HH}^*(kG)$ contains less extraneous information than does $HH^*(kG)$.

\begin{prop}\label{TateZero}Suppose $G$ is a group with conjugacy class representatives $\{ g_i \}$ and corresponding centralizers $H_i = C_G(g_i)$. If $r$ is the number of indices $i$ for which $p \mid |H_i|$, then $\dim \widehat{HH}^0(kG) = r$. Moreover, $\bar{Z}^{\text{pr}}(kG) = \Tr_1^G(kG)$.\end{prop}
\begin{proof}
It suffices to show that $\dim \bar{Z}(kG) = r$. Recall that $kG \otimes kG^{\text{op}} \simeq k[G \times G]$ with $kG$ a $k[G \times G]$-module via the rule $(g,h)j = gjh^{-1}$. Notice that the map $\psi : k[G \times G] \twoheadrightarrow kG$ given by $\psi(g,h) = gh^{-1}$ is $k[G \times G]$-linear since
$\psi((g_1,h_1)(g_2,h_2)) = g_1g_2h_2^{-1}h_1^{-1} = (g_1,h_1).\psi(g_2,h_2)$. So $\psi$ is a projective cover, and hence an endomorphism of $kG$ factors through a projective $k[G \times G]$-module precisely when it factors through $\psi$.

Now suppose $z \in Z(kG)$ and define an endomorphism $f_z$ of $kG$ via $f_z(x) = xz = zx$ for $x \in kG$. We aim to find a necessary and sufficient condition that $f_z$ factor through $\psi$, in the sense that $f_z = \psi\sigma$ for some $\sigma : kG \rightarrow k[G \times G]$. Since $kG \simeq k_{\Delta G}\uparrow^{G \times G}$ as $k[G \times G]$-modules, we know that for every $k\Delta G$-linear map $k \rightarrow k[G \times G]$ there is a unique $k[G \times G]$-linear extension $kG \rightarrow k[G \times G]$. So $\sigma$ is specified by $\sigma(1) = \sum_{(h,j)} \sigma_{(h,j)} (h,j)$ subject to $(g,g)\sigma(1) = \sigma(1)$ for all $g \in G$. In other words, $\sigma_{h,j} = \sigma_{gh,gj}$ for all $g,h,j \in G$. In particular, $\sigma$ is determined by $\{ \sigma_{g,1} : g \in G \}$ since $\sigma_{h,j} = \sigma_{j^{-1}h,1}$. From $f_z(1) = \psi(\sigma(1))$ we obtain $z = \sum \sigma_{h,j} hj^{-1}$. Let $\{g_i\}$ be conjugacy class representatives of $G$ and for each $g_i$ let $S_i = \{ {^g g_i} : g \in G \}$. Define $\kappa_i = \sum_{g \in S_i} g$ and write $z = \sum z_i \kappa_i$ for scalars $z_i \in k$. Then

\begin{equation}\label{zi}
z_i = \sum_{j \in G} \sigma_{g_i j,j} = \sum_{j \in G} \sigma_{j^{-1}g_i j,1} = |C_G(g_i)| \sum_{g \in S_i} \sigma_{g,1}
\end{equation}

So a necessary and sufficient condition for $f_z$ to factor through $\psi$ is the existence of solutions to (\ref{zi}) for all $i$. If (\ref{zi}) is a consistent system and $i$ is such that $p \mid |H_i|$, then $z_i = 0$, so that $z$ is a combination of all $\kappa_i$ with $p \nmid |H_i|$. Conversely, if $z$ is of this form then we may find a solution to (\ref{zi}) by setting, for example, $\sigma_{g,1} = 0$ if $p \mid |H_i|$ and $\sigma_{g,1} = z_i/|H_i|$ if $p \nmid |H_i|$. Therefore

$$\dim Z^{\text{pr}}(kG) = |\{ i : p \nmid |H_i| \}| = \dim Z(kG) - r$$

and so $\dim \bar{Z}(kG) = \dim Z(kG) - \dim Z^{\text{pr}}(kG) = r$ as claimed. The equality $\bar{Z}^{\text{pr}}(kG) = \Tr_1^G(kG)$ follows from the fact that $\Tr_1^G(g_i) = 0$ if $p \mid |C_G(g_i)|$ while $\Tr_1^G(\frac{1}{|C_G(g_i)|}g_i) = \kappa_i$ if $p \nmid |C_G(g_i)|$.
\end{proof}

Notice that $\widehat{HH}^*(kG) = \bigoplus \widehat{H}^*(H_i)$ as graded vector spaces where $\hat{H}^*(H_i)$ is the Tate cohomology of $H_i$, and that (\ref{ProductFormula}) still makes sense since restriction and trace are defined for Tate cohomology of groups. It is not clear whether the product formula (\ref{ProductFormula}) remains valid in this case, but even if it did, there is good reason to avoid computing products in negative degrees since these products may be complicated. Our original goal in introducing $\widehat{HH}^*(kG)$ was simply in discarding the extraneous information present in degree zero, and so we are justified in merely considering the subalgebra $HH'(kG) := \bigoplus_{n \geq 0} \widehat{HH}^n(kG)$. Note that $Z^{\text{pr}}(kG)$ is an ideal in $Z(kG)$ and $HH^*(kG)$ as well. After all, if $\alpha \in H^*(H_i)$ and $\beta \in H^*(H_j)$ with $\deg \beta > 0$ and $p \nmid |H_i|$, then $H^{\deg \beta}(W) = 0$ in the notation of (\ref{ProductFormula}), so that $\gamma_i(\alpha) \smile \gamma_j(\beta) = 0$. So $Z^{\text{pr}}(kG)$ is an ideal in $HH^*(kG)$, $HH'(kG) = HH^*(kG)/Z^{\text{pr}}(kG)$, and hence formula (\ref{ProductFormula}) remains valid for $HH'(kG)$.

\begin{remark}One of the original motivations for our computations was to investigate a conjecture formulated in \cite{Witherspoon} regarding a relationship between $HH^*(kG)$ and $H^*(G)$. More precisely, let $e_0$ be the principal block idempotent of $kG$ with $b_0 = e_0kG$, and recall that $HH^*(b_0)$ is an ideal direct summand of $HH^*(kG)$ with $HH^*(b_0) = e_0HH^*(kG)$. Then the composite map $H^*(G) \rightarrow HH^*(kG) \twoheadrightarrow HH^*(b_0)$ is a monomorphism with retraction $HH^*(b_0) \hookrightarrow HH^*(kG) \rightarrow H^*(G)$, where the map $HH^*(kG) = H^*(G,kG) \rightarrow H^*(G)$ is induced by $\varepsilon : kG \rightarrow k$. It was conjectured that the composite is an isomorphism module radicals; a result that was later established by Linckelmann in \cite{Linckelmann}. In this regard, our modification $HH'(kG)$ is harmless since the monomorphism $H^*(G) \hookrightarrow HH'(kG)$ remains an isomorphism modulo radicals. After all, the only thing to check is that $Z(b_0)/Z^{\text{pr}}(b)$ and $k$ are isomorphic modulo radicals, which is true since $J(Z(b_0))$ has codimension 1 in $Z(b_0)$.
\end{remark}

The next two results will prove indispensable, where the first is a standard application of the LHS.

\begin{lemma}\label{LHS}Suppose $G$ is a group with subgroup $N \triangleleft G$ and $k$ is a field of characteristic $p$. If $p \nmid |N|$ then $H^*(G) \simeq H^*(G/N)$. If $p \nmid |G:N|$ then $\Res : H^*(G) \simeq H^*(N)^G$ is an isomorphism with inverse $\frac{1}{|G:N|}\Tr_N^G$. Moreover, if we identify $H^*(G)$ with $H^*(N)^G$ and let $\alpha \in H^*(N)$, then $\Tr_N^G(\alpha) = \sum_{g \in G/N} g^*\alpha$.\end{lemma}

\begin{lemma}\label{ResComp}Suppose $G$ has a Sylow $p$-subgroup $P$ and a normal $p$-complement $Q$. If $k$ is a field with characteristic $p$, then $\Res : H^*(G) \rightarrow H^*(P)$ is an isomorphism.\end{lemma}
\begin{proof}
If $e = \frac{1}{|Q|}\sum_{q \in Q} q$ so that $ekG$ is the principal block of $kG$. Then $H^*(kG) \simeq \Ext_{ekG}^*(k,k) \simeq H^*(P)$ since $ekG \simeq kP$. Since $\Res : H^*(G) \rightarrow H^*(P)$ is injective, we see that $\Res$ is an isomorphism.
\end{proof}

\section{Linckelmann's Result}

We now use the product formula to establish a modified form of Proposition 5.2 from \cite{LinckelmannFrobenius}.

\begin{thm}\label{Linckelmann}Suppose $P$ is a nontrivial abelian $p$-group and $E$ is a $p'$-subgroup of $\text{Aut}(P)$ that acts semiregularly on $P \setminus \{1\}$. Then $E$ acts on the algebra $kP \otimes H^*(P)$ diagonally via automorphisms, and there is an isomorphism of $k$-algebras

$$HH'(k(E \ltimes P)) \simeq (kP \otimes H^*(P))^E$$

\end{thm}

\begin{proof}
Define $G = E \ltimes P$ and note that if $H \leq G$ then $H \cap P$ is a Sylow $p$-subgroup of $H$, and hence $\Res : H^*(H) \rightarrow H^*(P)^{E_H}$ is an isomorphism by Lemma \ref{LHS}, where $E_H$ is a $p$-complement in $H$. Of course $E_G = E$ and so $H^*(G) \simeq H^*(P)^E$. For $g \in G$ write $g = pe$ with $p \in P$ and $e \in E$. If $p_1 \in P$ then $p_1g = pp_1e$ since $P$ is abelian, and $gp_1 = p{^e p_1}e$, so that $p_1g = gp_1$ precisely when $p_1 = {^e p_1}$. That is, $C_G(g) \cap P = C_G(e) \cap P$. Since $E$ acts semiregularly on $P \setminus \{1\}$, $C_G(e) \cap P = 1$ for $e \not= 1$ and of course $C_G(1) \cap P = P$. Therefore, $C_G(g) \cap P = 1$ for $g \not\in P$ and $C_G(g) = P$ for $1 \not= g \in P$. So if $G$ has conjugacy class representatives $\{g_i\}$, then $H^*(g_i) = k$ in degree zero whenever $g_i \not\in P$, and $H^*(g_i) = H^*(P)$ whenever $g_i \in P$. Also, the set of all $g_i$ that belong to $P$ comprise the representatives of the orbits of $E$ acting on $P$. The additive decomposition for $HH'(kG)$ now becomes

$$HH'(kG) = \bigoplus_{p \nmid |H_i|} H^*(C_G(g_i)) = H^*(P)^E \oplus \bigoplus_{1 \not= g_i \in P} H^*(P)$$

Define $\Psi : HH'(kG) \rightarrow kP \otimes H^*(P)$ by

$$\Psi(\gamma_i(\alpha)) = \begin{cases}
1 \otimes \alpha &\text{if $g_i = 1$}\\
\sum_{e \in E}{^e g_i} \otimes e^*\alpha &\text{if $1 \not= g_i \in P$}
\end{cases}$$

Clearly $\Psi(\gamma_1(\alpha)) \in (kP \otimes H^*(P))^E$ for $\alpha \in H^*(P)^E$, and also $\Psi(\gamma_i(\alpha)) \in (kP \otimes H^*(P))^E$ for $\alpha \in H^*(P)$ since

$$^f\Psi(\gamma_i(\alpha)) = \sum_{e \in E} {^{fe} g_i} \otimes (fe)^*\alpha = \sum_{e \in E} {^e g_i} \otimes e^*\alpha = \Psi(\gamma_i(\alpha))$$

whenever $f \in E$. It is clear that $\Psi$ is injective. Notice that $kP \otimes H^*(kP)$ is graded by taking $kP$ concentrated in degree zero. Since $E$ preserves the grading of $kP \otimes H^*(kP)$, an element $\xi \in kP \otimes H^*(kP)$ is fixed under $E$ if and only if each of its homogeneous components is fixed under $E$. Also, any homogeneous element $\xi$ of $kP \otimes H^*(kP)$ of degree $n$ may be expressed as

$$\xi = 1 \otimes \alpha + \sum_{1 \not= g_i \in P} \sum_{e \in E}{^e g_i} \otimes \alpha_{i,e}$$

for some $\alpha, \alpha_{i,e} \in H^*(kP)$. It is clear that $\xi$ is fixed under $E$ precisely when $\alpha \in H^*(P)^E$ and $f^*\alpha_{i,e} = \alpha_{i,fe}$ for all $e,f \in E$. In other words, $\alpha_{i,e} = e^*\alpha_{i,1}$ for all $e \in E$. Therefore, $\Psi$ is a surjective map, and it remains to verify that $\Psi$ is multiplicative. It suffices to show that $\Psi(\gamma_i(\alpha)\gamma_j(\beta)) = \Psi(\gamma_i(\alpha))\Psi(\gamma_j(\beta))$. This is clear if $i$ or $j$ equals 1 since $\gamma_1(\alpha)\gamma_i(\beta) = \gamma_i(\Res(\alpha)\beta)$. So assume $1 \not= g_i, g_j \in P$ and $\alpha, \beta \in H^*(P)$. Note that $D = E$ in the notation of (\ref{ProductFormula}) and so for every $e \in E$ there is a summand $\gamma_k( \cdot )$ in the product $\gamma_i(\alpha)\gamma_j(\beta)$, and hence a corresponding summand $\Psi(\gamma_k(\cdot))$ in $\Psi(\gamma_i(\alpha)\gamma_j(\beta))$. More precisely, for $e \in E$ there is $e' \in E$ and $g_k \in P$ with $g_k = {^{e'}(g_i{^e g_j})}$. If $g_k \not= 1$ then the summand in $\gamma_i(\alpha)\gamma_j(\beta)$ equals $\gamma_k((e')^*\alpha \smile (e'e)^*\beta)$ and so the corresponding summand in $\Psi(\gamma_i(\alpha)\gamma_j(\beta))$ equals

\begin{equation}\label{Sum1}
\begin{split}
\Psi(\gamma_k((e')^*\alpha \smile (e'e)^*\beta))
&= \sum_{f \in E} {^f ({^{e'}(g_i{^e g_j})})} \otimes f^*((e')^*\alpha \smile (e'e)^*\beta)\\
&= \sum_{f \in E} {^f(g_i{^e g_j})} \otimes (f^*\alpha \smile (fe)^*\beta)\\
\end{split}
\end{equation}

since $fe'$ ranges across $E$ as $f$ ranges across $E$. On the other hand, if $g_k = 1$ then the summand in $\gamma_i(\alpha)\gamma_j(\beta)$ equals $\gamma_1(\Tr_P^G(\alpha \smile e^*\beta)) = \sum_{f \in E} \gamma_1(f^*(\alpha \smile e^*\beta))$ by Lemma \ref{LHS}. So the corresponding summand in $\Psi(\gamma_i(\alpha)\gamma_j(\beta))$ equals

\begin{equation}\label{Sum2}
\sum_{f \in E} {^f(g_i{^e g_j})} \otimes (f^*\alpha \smile (fe)^*\beta)
\end{equation}

since ${^f (g_i {^e g_j}}) = {^f 1} = 1$ for $f \in E$. As $(f,e)$ ranges across $E \times E$ so does $(e_1,e_2) := (f,fe)$. So summing (\ref{Sum1}) and (\ref{Sum2}) across $e \in E$ yields

\begin{equation*}
\begin{split}
\Psi(\gamma_i(\alpha)\gamma_j(\beta))
&= \sum_{f \in E} \sum_{e \in E} {^f(g_i{^{e}g_j})} \otimes (f^*\alpha \smile (fe)^*\beta) \\
&= \sum_{e_1,e_2 \in E} {^{e_1}g_i}{^{e_2}g_j} \otimes (e_1^*\alpha \smile e_2^*\beta) = \Psi(\gamma_i(\alpha)) \Psi(\gamma_j(\beta))
\end{split}
\end{equation*}

as desired. So the proof is complete.
\end{proof}

\section{Groups of Order less than 16}

We may now 'complete' the computation of $HH'(kG)$ for $|G| \leq 15$ begun in \cite{Witherspoon}. We simply need to compute $HH'$ for $\F_2 D_{10}$, $\F_3 A_4$, $\F_2 D_{12}$, $\F_3 D_{12}$, $\F_2 T$, $\F_3 T$, and $\F_2 D_{14}$, where $T$ denotes the nonabelian group of order 12 not isomorphic with $D_{12}$ or $A_4$. Recall that $HH'(kG) = \bigoplus_{p \mid |H_i|} H^*(H_i)$ where $H_i = C_G(g_i)$. To describe a finite presentation for $HH'(kG)$ we follow \cite{Witherspoon} by listing homogeneous elements $X_1,\ldots,X_m$ that generate $H^*(G) \simeq \gamma_1(H^*(G))$ as an algebra subject to the homogenous relations $r_1,\ldots,r_n$; homogenous elements $Y_1,\ldots,Y_{m'} \in \bigoplus_{i \geq 2} H^*(H_i)$ that generate $\bigoplus_{i \geq 2} H^*(H_i)$ as a $\gamma_1(H^*(G))$-module subject to relations $r_1,\ldots,r_{n'}$; and relations of the form $Y_iY_j = X_{ij} + \sum X_{ij}^k Y_k$ for all $i$ and $j$ where $X_{ij},X_{ij}^k \in \gamma_1(H^*(G))$. We refer to these as relations of Type I, Type II, and Type III respectively. Together with the (implicitly understood) graded-commutative relations, we obtain an abstract presentation of the algebra $HH'(kG)$. In performing these computations it is convenient to identify $HH'(kG)$ in degree zero with $\bar{Z}(kG)$.

\begin{prop}For $n$ odd, $HH'(\F_2 D_{2n})$ has generators $C_1, C_2$, and $V$ of degrees 0,0, and 1, respectively, subject to the relation $(C_2)^2 = C_1$.\end{prop}
\begin{proof}
Write $n = 2l+1$ and $D_{2n} = \lr{a,b | a^n = b^2 = 1, bab^{-1} = a^{-1} }$. Note that $D_{2n}$ has conjugacy class representatives $\{ g_i \} = \{ 1,a,\ldots,a^l,b \}$ with corresponding centralizers $\{H_i\} = \{ D_{2n}, \lr{a}, \ldots, \lr{a}, \lr{b} \}$. Lemma \ref{ResComp} implies that $\Res : H^*(D_{2n}) \rightarrow H^*(\lr{b})$ is an isomorphism, and so $H^*(D_{2n})$ is generated by an element $v$ of degree 1. Therefore, $HH'(\F_2 D_{2n})$ has generators $C_1 = \gamma_1(1), C_2 = \gamma_2(1)$, and $V = \gamma_1(v)$ of degrees 0,0,and 1, respectively. The only relations are of Type III in degree zero. Identifying $HH'^0(kG)$ with $\bar{Z}(kG)$ and defining $\tau = \sum_{i=0}^{n-1} a^i$, this relation takes the form

$$(C_2)^2 = (\tau b + Z^{\text{pr}}(kG))^2 = \tau^2 + Z^{\text{pr}}(kG) = \tau + Z^{\text{pr}}(kG) = 1 + Z^{\text{pr}}(kG) = C_1$$

So we have a complete presentation of $HH'(kG)$.
\end{proof}

\begin{prop}$HH'(\F_3 A_4)$ has generators $C_1,C_2,C_3,V,$ and $U$ of degrees 0,0,0,1, and 2, respectively, subject to the relations $(C_2)^2 = C_3, (C_3)^2 = C_2$, and $C_2C_3 = C_3C_2 = C_1$.\end{prop}
\begin{proof}
Let $a = (12)(34), b = (13)(24)$, and $c = (123)$. So $A_4$ has conjugacy class representatives $\{ g_i \} = \{ 1, c, c^2, a \}$ with corresponding centralizers $\{ H_i \} = \{ A_4, \lr{c}, \lr{c}, P \}$ where $P = \lr{a,b}$ is a normal Sylow 2-subgroup of $A_4$. Lemma \ref{ResComp} implies that $\Res : H^*(A_4) \rightarrow H^*(\lr{c})$ is an isomorphism. So $H^*(A_4) = \F_3[u,v]$ where $\deg u = 2$ and $\deg v = 1$. Of course, $H^*(\lr{c})$ is free as an $H^*(A_4)$-module with basis 1. Hence, $HH'(\F_3 A_4)$ has generators $C_1 = \gamma_1(1), C_2 = \gamma_2(1), C_3 = \gamma_3(1), V = \gamma_1(v)$, and $U = \gamma_1(u)$ of degrees 0,0,0,1, and 2, respectively. There are no nontrivial Type I or Type II relations. The Type III relations only occur in degree zero, in which case they are $(C_2)^2 = C_3, (C_3)^2 = C_2$, and $C_2C_3 = C_1$, as is straightforward to check.
\end{proof}

\begin{thm}$HH'(\F_2 D_{12})$ has generators $C_1$,$C_2$,$C_3$,$C_4$,$C_5$,$C_6$,$V_1$, and $V_2$ of degrees 0,0,0,0,0,0,1, and 1, respectively, subject to the relations $V_2C_3 = V_2C_4 = 0$ and those given by the table below:

\begin{table}[here]
\begin{center}
\begin{tabular}{c|ccccc}
  & $C_2$ & $C_3$ & $C_4$ & $C_5$ & $C_6$ \\
  \hline
  $C_2$ & $C_1$ & $C_4$ & $C_3$ & $C_6$ & $C_5$ \\
  $C_3$ & $C_4$ & $C_4$ & $C_3$ & 0 & 0 \\
  $C_4$ & $C_3$ & $C_3$ & $C_4$ & 0 & 0 \\
  $C_5$ & $C_6$ & 0 & 0 & $C_1+C_4$ & $C_2+C_3$ \\
  $C_6$ & $C_5$ & 0 & 0 & $C_2+C_3$ & $C_1+C_4$
\end{tabular}
\end{center}
\end{table}

Moreover, $HH'(\F_3 D_{12})$ has generators $C_1,C_2,C_3,C_4,X_3,X_4,Y_3,Y_4,V$, and $U$ of degrees 0,0,0,0,1,1,2,2,3, and 4, respectively, subject to the Type II relations $VX_3 = VX_4 = 0$, $UX_3 = VY_3$, and $UX_4 = VY_4$; Type III relations in degree zero given as $(C_2)^2 = C_1$, $(C_3)^2 = C_4-C_1$, $(C_4)^2 = C_4-C_1$, $C_2C_3 = C_4$, $C_2C_4 = C_3$, and $C_3C_4 = C_3-C_2$; and the Type III relations in positive degree given in the following table.

\begin{table}[here]
\begin{center}
\begin{tabular}{c|cccc}
  & $X_3$ & $Y_3$ & $X_4$ & $Y_4$ \\
  \hline
  $C_2$ & $2X_4$ & $2Y_4$ & $2X_3$ & $2Y_3$ \\
  $C_3$ & $X_4$ & $Y_4$ & $X_3$ & $Y_3$ \\
  $C_4$ & $2X_3$ & $2Y_3$ & $2X_4$ & $2Y_4$ \\
  $X_3$ & 0 & $VC_4+V$ & 0 & $2VC_2+2VC_3$ \\
  $Y_3$ & $VC_4+V$ & $UC_4+U$ & $2VC_2+2VC_3$ & $2UC_2+2UC_3$ \\
  $X_4$ & 0 & $2VC_2+2VC_3$ & 0 & $VC_4+V$ \\
  $Y_4$ & $2VC_2+2VC_3$ & $2UC_2+2UC_3$ & $VC_4+V$ & $UC_4+U$ \\
\end{tabular}
\end{center}
\end{table}

\end{thm}
\begin{proof}
Let $G = D_{12} = \lr{a,b | a^6 = b^2 = 1, bab^{-1} = a^{-1}}$ and define $N = \lr{a}$ and $Z = Z(G) = \lr{a^3}$. So $G$ has conjugacy class representatives $\{g_i\} = \{1,a^3,a,a^2,b,ab\}$ with corresponding centralizers $\{G, G, N, N, \lr{a^3,b}, \lr{a^3,ab}\}$ and orders $\{12,12,6,6,4,4\}$.

$(\textit{p = 2})$ Since $\lr{a^2}$ is a normal 2-complement, Lemma \ref{ResComp} implies that $\Res : H^*(G) \rightarrow H^*(P)$ is an isomorphism whenever $P \in \Syl_2(G)$, and in particular $H^*(P)$ is free as an $H^*(G)$-module with basis 1. Fix $P = \lr{a^3,b}$ and choose generators $v_1$ and $v_2$ of $H^*(G)$ with degrees 1 and 1, so that $\Res^G_P(v_1)$ and $\Res^G_P(v_2)$ are 'dual' to $a^3$ and $b$, in the sense that $\Res^G_Z(v_1) \not= 0$ and $\Res^G_Z(v_2) = 0$. From $\Res^N_Z \Res^G_N = \Res^P_Z \Res^G_P$ and $\Tr_Z^N \Res^N_Z = |N:Z|\Id = \Id$ we obtain $\Res^G_N = \Tr_Z^N \Res^P_Z \Res^G_P$. Since $\Tr_Z^N$ is an isomorphism by Lemma \ref{ResComp}, $\Res^G_N(v_1)$ generates $H^*(N)$ as an algebra, and $\Res^G_N(v_2) = 0$. In particular, $H^*(N)$ is generated by 1 as an $H^*(G)$-module. Therefore, we have generators of $HH'(\F_2 D_{12})$ given by $C_i = \gamma_i(1)$ for $1 \leq i \leq 6$, and $V_j = \gamma_1(v_j)$ for $j=1,2$. There are no nontrivial Type I relations, Type III relations occur only in degree zero, and we have the relations $V_2C_3 = V_2C_4 = 0$ of Type II. If we let $H^*(G)[n]$ denote the free $H^*(G)$-module with grading $H^i(G)[n] = H^{i-n}(G)$, then we have a graded complex of $H^*(G)$-modules given by

$$\xymatrix{H^*(G)[1] \ar[r]^{\delta_1} & H^*(G) \ar[r]^{\delta_0} & H^*(N) \ar[r] & 0}$$

where $\delta_0(\alpha) = \alpha.1$ and $\delta_1(\alpha) = \alpha v_2$ for $\alpha \in H^*(G)$. Note that $\delta_0$ is surjective and $\delta_1$ is injective since $H^*(G)$ is an integral domain. Since $\dim H^i(G)[1]-\dim H^i(G)+\dim H^i(N) = 0$ for all $i \geq 0$, we see that this complex is exact, and hence we have a free presentation of $H^*(N)$ as an $H^*(G)$-module. In particular, we have accounted for all Type II relations.

$(\textit{p = 3})$ Lemma \ref{LHS} implies that $\Res : H^*(G) \rightarrow H^*(N)^G$ is an isomorphism. So there are generators $x$ and $y$ of $H^*(N)$ with degrees 1 and 2, and generators $v$ and $u$ of $H^*(G)$ with degrees 3 and 4, such that restriction $H^*(G) \rightarrow H^*(N)$ sends $v$ to $xy$ and $u$ to $y^2$. In particular, $H^*(N)$ is generated as an $H^*(G)$-module by $1,x$, and $y$, with relations $v.x = 0$ and $u.x = v.y$. So letting $B[n]_* = H^*(G)[n]$ denote the free graded $H^*(G)$-module of rank 1, we have a graded complex of $H^*(G)$-modules given by

$$\xymatrix{B[4] \oplus B[5] \ar[r]^(.4){\delta_1} & B \oplus B[1] \oplus B[2] \ar[r]^(.6){\delta_0} & H^*(N) \ar[r] & 0}$$

where

\begin{equation*}
\begin{split}
\delta_0(\alpha_1,\alpha_2,\alpha_3) &= \alpha_1.1+\alpha_2.x+\alpha_3.y \\
\delta_1(\alpha_1,\alpha_2) &= \alpha_1(0,v,0)+\alpha_2(0,u,-v)
\end{split}
\end{equation*}

Clearly $\delta_0$ is surjective. Suppose $(\alpha_1,\alpha_2) \in (B[4] \oplus B[5])_i$ and $\delta_1(\alpha_1,\alpha_2) = 0$, so that $\alpha_2v = 0$ and $\alpha_2 u + \alpha_1 v = 0$. Homogeneous elements of even order are non-zero divisors in $H^*(G)$, whereas odd degree elements are only annihilated by other odd degree elements. So $\alpha_2 = \lambda vu^j$ for some $\lambda \in k$ and $j \geq 0$. Also, from $0 = v(\alpha_1 + \lambda u^{j+1})$ we obtain $\alpha_1 = -\lambda u^{j+1}$ and so $i = 4j+8$. Therefore, $\Ker(\delta_1) = \bigoplus_{j=0}^{\infty} \Ker(\delta_1|(B[4] \oplus B[5])_{4j+8})$ where $\Ker(\delta_1|(B[4] \oplus B[5])_{4j+8})$ is spanned by $(-u^{j+1},vu^j)$. Using the fact that $\dim B[n]_t$ has Poincar\'{e} series $q_n(t)$ given as

$$q_n(t) = t^n(1+t^3+t^4+t^7+t^8+\cdots)$$

it is easy to see that the sequence is exact, and hence we have found a free presentation of $H^*(N)$ as an $H^*(G)$-module. Thus, $HH'(\F_3 D_{12})$ has generators given by $C_i = \gamma_i(1)$ for $1 \leq i \leq 4$, $V = \gamma_1(v)$, $U = \gamma_1(u)$, and $X_j = \gamma_j(x)$ and $Y_j = \gamma_j(y)$ for $j = 3, 4$. There are no nontrivial Type I relations. The Type II relations are $VX_3 = VX_4 = 0$, $UX_3 = VY_3$, and $UX_4 = VY_4$. Type III relations in degree zero are obtained in the usual fashion, and Type III relations in positive degrees are obtained from the rules:

\begin{equation*}
\begin{split}
\gamma_3(\alpha) \smile \gamma_3(\beta) &= \gamma_4(\alpha \smile \beta) + \gamma_1(\Tr_N^G(\alpha \smile b^*\beta)) \\
\gamma_4(\alpha) \smile \gamma_4(\beta) &= \gamma_4(b^*\alpha \smile b^*\beta) + \gamma_1(\Tr_N^G(\alpha \smile b^*\beta)) \\
\gamma_3(\alpha) \smile \gamma_4(\beta) &= \gamma_2(\Tr_N^G(\alpha \smile \beta)) + \gamma_3(b^*\alpha \smile \beta))
\end{split}
\end{equation*}

Note that $b^*x = x$ and $b^*y = y$; this is how we computed $H^*(G) \simeq H^*(N)^G$. So we have a complete presentation of $HH'(\F_3 D_{12})$.
\end{proof}

\newpage

\begin{thm}Suppose $T = \Z_4 \ltimes \Z_3$. Then $HH'(\F_2 T)$ has generators $C_1$, $C_2$, $C_3$, $C_4$, $C_5$, $C_6$, $U$, $V$, $X_1$ and $X_2$ of degrees 0,0,0,0,0,0,1,2,1, and 1, respectively, subject to the Type I relation $V^2 = 0$, the Type II relations $VC_3 = VC_4 = VX_1 = VX_2 = 0$, the Type III relations in degree zero given as follows:

\begin{table}[here]
\begin{center}
\begin{tabular}{c|ccccccc}
  & $C_2$ & $C_3$ & $C_4$ & $C_5$ & $C_6$ \\
  \hline
  $C_2$ & $C_1$ & $C_4$ & $C_3$ & $C_6$ & $C_5$ \\
  $C_3$ & $C_4$ & $C_3$ & $C_4$ & 0 & 0 \\
  $C_4$ & $C_3$ & $C_4$ & $C_3$ & 0 & 0 \\
  $C_5$ & $C_6$ & 0 & 0 & $C_2+C_4$ & $C_1+C_3$ \\
  $C_6$ & $C_5$ & 0 & 0 & $C_1+C_3$ & $C_2+C_4$ \\
\end{tabular}
\end{center}
\end{table}

and the Type III relations in positive degrees given by following table:

\begin{table}[here]
\begin{center}
\begin{tabular}{c|ccccccc}
  & $C_2$ & $C_3$ & $C_4$ & $C_5$ & $C_6$ & $X_3$ & $X_4$ \\
  \hline
  $X_3$ & $X_4$ & $X_3+V$ & $X_4$ & $VC_5$ & $VC_6$ & $UC_3$ & $UC_4$ \\
  $X_4$ & $X_3$ & $X_3+V$ & $X_4$ & $VC_6$ & $VC_5$ & $UC_4$ & $UC_3$ \\
\end{tabular}
\end{center}
\end{table}

Moreover, $HH'(\F_3 T)$ has generators $C_1,C_2,C_3,C_4,U,V,X_1,X_2,Y_1$ and $Y_2$ of degrees 0,0,0,0,3,4,1,2,1, and 2, respectively, subject to the Type III relations in degree zero given as follows:

\begin{table}[here]
\begin{center}
\begin{tabular}{c|ccccccc}
  & $C_2$ & $C_3$ & $C_4$ \\
  \hline
  $C_2$ & $C_1$ & $C_4$ & $C_3$ \\
  $C_3$ & $C_4$ & $2C_1+C_3$ & $2C_2+C_4$ \\
  $C_4$ & $C_3$ & $2C_2+C_4$ & $2C_1+C_3$ \\
\end{tabular}
\end{center}
\end{table}

and the Type III relations in positive degrees given as

\begin{table}[here]
\begin{center}
\begin{tabular}{c|ccccccc}
  & $C_2$ & $C_3$ & $C_4$ & $X_3$ & $Y_3$ & $X_4$ & $Y_4$ \\
  \hline
  $X_3$ & $X_4$ & $2X_3$ & $2X_4$ & 0 & $VC_3+V$ & 0 & $VC_4+VC_2$ \\
  $Y_3$ & $Y_4 $ & $2Y_3$ & $2Y_4$ & $VC_3+V$ & $UC_3+U$ & $VC_4+VC_2$ & $UC_4+UC_2$ \\
  $X_4$ & $X_3$ & $2X_4$ & $2X_3$ & 0 & $VC_4+VC_2$ & 0 & $VC_3+V$ \\
  $Y_4$ & $Y_3$ & $2Y_4$ & $2Y_3$ & $VC_4+VC_2$ & $UC_4+UC_2$ & $VC_3+V$ & $UC_3+U$ \\
\end{tabular}
\end{center}
\end{table}

\end{thm}

\begin{proof}
Write $T = \lr{a,b|a^4=b^3=1,aba^{-1}=b^{-1}}$ and $N = \lr{a^2,b}$ so that $T$ has conjugacy class representatives $\{g_i\} = \{1,a^2,b,a^2b,a,a^3b\}$ with corresponding centralizers $\{H_i\} = \{T,T,N,N,\lr{a},\lr{a^3b}\}$ with orders $\{|H_i|\} = \{12,12,6,6,4,4\}$.

$(\textit{p = 2})$ Lemma \ref{ResComp} implies that $\Res : H^*(T) \rightarrow H^*(P)$ is an isomorphism whenever $P \in \Syl_2(T)$. So $H^*(T)$ is generated by elements $v$ and $u$ of degrees 1 and 2, respectively, subject to the relation $v^2 = 0$. Also, $H^*(N) \simeq H^*(\lr{a^2})$ is generated by an element $x$ with degree 1. As usual, $H^*(P)$ is the free $H^*(T)$-module with basis 1. Since $\Res : H^*(\lr{a}) \rightarrow H^*(\lr{a^2})$ is zero in odd degrees and nonzero in even degrees, we see that $\Res : H^*(T) \rightarrow H^*(N)$ maps $v$ to 0 and $u$ to $x^2$. Hence, $H^*(N)$ is generated as an $H^*(T)$-module by $1$ and $x$. So letting $B[n]_* = H^*(T)[n]$ denote the free graded $H^*(T)$-module of rank 1, we have a graded complex of $H^*(T)$-modules given by

$$\xymatrix{B[1] \oplus B[2] \ar[r]^{\delta_1} & B[0] \oplus B[1] \ar[r]^(.6){\delta_0} & H^*(N) \ar[r] & 0}$$

where

\begin{equation*}
\begin{split}
\delta_0(\alpha_1,\alpha_2) &= \alpha_1.1+\alpha_2.x \\
\delta_1(\alpha_1,\alpha_2) &= \alpha_1(v,0)+\alpha_2(0,v)
\end{split}
\end{equation*}

Clearly, $\delta_0$ is surjective and $\delta_1(\alpha_1,\alpha_2) = 0$ precisely when $\alpha_1$ and $\alpha_2$ have odd degree. So if $q(t) = \sum_{i=0}^{\infty} t^i$ then $B[0] \oplus B[1]$ has corresponding Poincar\'{e} polynomial $q(t)+tq(t)$, $\text{Im}(\delta_1)$ has polynomial $tq(t)$, and $H^*(N)$ has polynomial $q(t)$, so that the above sequence is a free presentation of $H^*(N)$ as an $H^*(T)$-module. Therefore, we can choose generators of $HH'(kG)$ by $C_i = \gamma_i(1)$ for $1 \leq i \leq 6$, $V = \gamma_1(v)$, $U = \gamma_1(u)$, and $X_j = \gamma_j(x)$ for $j = 3,4$. The only Type I relation is $V^2 = 0$, the Type II relations are $VC_3 = VC_4 = VX_3 = VX_4 = 0$, and the Type III relations in degree zero are handled in the usual way. The Type III relations in positive degree are obtained by using (\ref{ProductFormula}). For instance, we have the following:

$$\gamma_3(\alpha) \smile \gamma_2(\beta) = \gamma_4(\alpha \smile \Res^T_H(\beta))$$
$$\gamma_3(\alpha) \smile \gamma_3(\beta) = \gamma_3(a^*\alpha \smile a^*\beta) + \gamma_1(\Tr_N^T(\alpha \smile a^*\beta))$$
$$\gamma_3(\alpha) \smile \gamma_4(\beta) = \gamma_4(a^*\alpha \smile a^*\beta) + \gamma_2(\Tr_N^T(\alpha \smile a^*\beta))$$
$$\gamma_4(\alpha) \smile \gamma_4(\beta) = \gamma_3(a^*\alpha \smile a^*\beta) + \gamma_1(\Tr_N^T(\alpha \smile a^*\beta))$$

From $\Tr_N^T = \Tr_P^T \Tr_Z^P \Res^N_Z$ we obtain $\Tr_N^T(1) = \Tr_N^T(y) = 0$ and $\Tr_N^T(x) = v$. Also, $a^*\alpha = \alpha$ for $\alpha \in H^*(Z)$ since $Z \leq ZT$. The computations are now laborious but straightforward.

$(\textit{p = 3})$ Lemma \ref{LHS} implies that $\Res : H^*(T) \rightarrow H^*(N)^T$ is an isomorphism. So $H^*(T)$ is generated by elements $v$ and $u$ of degrees 3 and 4, respectively. We may choose generators $x$ and $y$ of $H^*(N)$ with degrees 1 and 2, respectively, for which $\Res(v) = xy$ and $\Res(u) = y^2$. Then $H^*(N)$ is generated by 1, $x$, and $y$ as an $H^*(T)$-module, with relations $v.x = 0$ and $u.x = v.y$. In fact, this provides a free presentation of $H^*(N)$ as an $H^*(T)$-module. We obtain generators of $HH'(kG)$ by defining $C_i = \gamma_i(1)$ for $1 \leq i \leq 4$, $V = \gamma_1(v)$, $U = \gamma_1(u)$, $X_j = \gamma_j(x)$, and $Y_j = \gamma_j(y)$ for $j = 3,4$. There are no nontrivial Type I relations, and the Type II relations are given by $VX_3 = VX_4 = 0$, $UX_3 = VY_3$, and $UX_4 = VY_4$. Type III relations are obtained by the same means as for $p = 2$. In fact, (\ref{ProductFormula}) remains valid regardless of the characteristic of the field. For these computations it useful to have the following table at hand:

\begin{table}[here]
\begin{center}
\begin{tabular}{c|ccccc}
  $\xi$ & 1 & $x$ & $y$ & $xy$ & $y^2$ \\
  \hline
  $\Tr_N^T(\xi)$ & 2 & 0& 0 & $2v$ & $2u$ \\
\end{tabular}
\end{center}
\end{table}

The proof is now complete.
\end{proof}

We include one last computation which is interesting in that it requires few detailed multiplications.

\begin{prop}$HH'(\F_3 \text{SL}_2(3))$ has generators $C_1,C_2,C_3,C_4,C_5,C_6,V$, and $U$ of degrees 0,0,0,0,0,0,1, and 2, respectively, subject to the relations given by the following table:

\begin{table}[h!b!p!]
\begin{center}
\begin{tabular}{c|ccccc}
  & $C_2$ & $C_3$ & $C_4$ & $C_5$ & $C_6$ \\
  \hline
  $C_2$ & $C_1$ & $C_6$ & $C_5$ & $C_4$ & $C_3$ \\
  $C_3$ & $C_6$ & $C_4$ & $C_1$ & $C_2$ & $C_5$ \\
  $C_4$ & $C_5$ & $C_1$ & $C_3$ & $C_6$ & $C_2$ \\
  $C_5$ & $C_4$ & $C_2$ & $C_6$ & $C_3$ & $C_1$ \\
  $C_6$ & $C_3$ & $C_5$ & $C_2$ & $C_1$ & $C_4$ \\
\end{tabular}
\end{center}
\end{table}

\end{prop}
\begin{proof}
Recall that $|\text{SL}_2(3)| = 24$ and the Sylow 2-subgroup of $\text{SL}_2(3)$ is normal. Thus, Lemma \ref{ResComp} implies that $\Res : H^*(\text{SL}_2(3)) \rightarrow H^*(P)$ is an isomorphism for all $P \in \Syl_3(\text{SL}_2(3))$. In particular, $H^*(\text{SL}_2(3))$ has generators $v$ and $u$ of degrees 1 and 2, respectively. Direct computation shows that $\text{SL}_2(3)$ has conjugacy class representatives $\{g_i\}$ given as

\begin{equation*}
\begin{pmatrix}
1 & 0 \\
0 & 1
\end{pmatrix}
\hspace{10 pt}
\begin{pmatrix}
2 & 0 \\
0 & 2
\end{pmatrix}
\hspace{10 pt}
\begin{pmatrix}
0 & 1 \\
2 & 2
\end{pmatrix}
\hspace{10 pt}
\begin{pmatrix}
2 & 2 \\
1 & 0
\end{pmatrix}
\hspace{10 pt}
\begin{pmatrix}
0 & 1 \\
2 & 1
\end{pmatrix}
\hspace{10 pt}
\begin{pmatrix}
1 & 2 \\
1 & 0
\end{pmatrix}
\hspace{10 pt}
\begin{pmatrix}
0 & 1 \\
2 & 0
\end{pmatrix}
\end{equation*}

with centralizers $H_1 = H_2 = \text{SL}_2(3)$ and $H_i = \lr{g_i,2I}$ for $3 \leq i \leq 7$. So $H_i$ is abelian for $3 \leq i \leq 7$, $|H_i| = 6$ for $3 \leq i \leq 6$, and $|H_7| = 4$. In particular, if $P_i$ denotes a Sylow 3-subgroup of $H_i$ then $\Res : H^*(H_i) \rightarrow H^*(P_i)$ is an isomorphism for $2 \leq i \leq 6$, so that $H^*(H_i)$ is free as an $H^*(\text{SL}_2(3))$-module with basis 1. Therefore, $HH'(\F_3\text{SL}_2(3))$ has generators defined as $C_i = \gamma_i(1)$ for $1 \leq i \leq 6$, $V = \gamma_1(v)$, and $U = \gamma_1(u)$. The only nontrivial relations are of Type III in degree zero, which are straightforward to obtain.
\end{proof}





\section{Acknowledgments}

Most of this work formed a part of my Ph.D. thesis conducted at the University of Chicago under the direction of J. L. Alperin. I would like to thank Alperin for recommending these problems to me and for providing me with constant encouragement in their investigation. I also wish to thank the authors of MAGMA, a program that has proved quite useful in performing various mathematical computations.

\bibliographystyle{plain}
\bibliography{AAAbibfile}

\end{document}